\documentclass[11pt,reqno]{amsart}
\usepackage{amsmath}
\usepackage{amsthm}
\usepackage{amsfonts}
\usepackage{amssymb}
\usepackage{array}   % for \newcolumntype macro
\newcolumntype{L}{>{$}l<{$}}

\usepackage{pgf,tikz}
\usepackage{mathrsfs}
\usetikzlibrary{arrows}
\newcommand{\one}{\mathbf{1}}

\usepackage{eucal}
\usepackage{graphicx}
\usepackage{multirow}
\usepackage[all,knot]{xy}
\usepackage{m}
\nc{\ot}{\otimes}

\DeclareMathOperator{\Rep}{Rep}
\DeclareMathOperator{\Stab}{Stab}
\DeclareMathOperator{\Ising}{Ising}
\DeclareMathOperator{\sVec}{sVec}
\DeclareMathOperator{\Sem}{Sem}
\calclayout

\newcommand{\mcB}{\mathcal{B}}
\newcommand{\mcC}{\mathcal{C}}
\newcommand{\mcD}{\mathcal{D}}
\DeclareMathOperator{\Aut}{Aut}
\newcommand{\tY}{\tilde{Y}}
\newcommand{\tX}{\tilde{X}}
\newcommand{\pz}{(0)}
\newcommand{\po}{(1)}
\begin{document}
\title{Metaplectic Categories, Gauging and Property F}
\author{Paul Gustafson}
\email{paul.gustafson@gmail.com}
\author{Eric C. Rowell}
\email{rowell@math.tamu.edu}
\author{Yuze Ruan}
\email{yuze.ruan@gmail.com}

\thanks{The first two authors were partially supported by US NSF grant DMS-1664359.  Part of this work was carried out while E.C.R. was visiting BICMR, Peking University and he gratefully acknowledges the support of that institution.}
\begin{abstract} $N$-Metaplectic categories, unitary modular categories with the same fusion rules as $SO(N)_2$, are prototypical examples of weakly integral modular categories generalizing the model for the Ising anyons.  As such, a conjecture of the second author would imply that images of the braid group representations associated with metaplectic categories are finite groups, i.e. have property $F$.  While it was recently shown that $SO(N)_2$ itself has property $F$, proving property $F$ for the more general class of metaplectic modular categories is an open problem. We verify this conjecture for $N$-metaplectic modular categories when $N$ is odd, exploiting their classification and enumeration to relate them to $SO(N)_2$.   In another direction, we prove that when $N$ is divisible by $8$ the $N$-metaplectic categories have $3$ non-trivial bosons, and the boson condensation procedure applied to 2 of these bosons yields $\frac{N}{4}$-metaplectic categories.  Otherwise stated: any $8k$-metaplectic category is a $\Z_2$-gauging of a $2k$-metaplectic category, so that the $N$ even metaplectic categories lie towers of $\Z_2$-gaugings commencing with $2k$- or $4k$-metaplectic categories with $k$ odd.
\end{abstract}
\maketitle
\section{Introduction}
$N$-Metaplectic  categories are a major source of examples of weakly integral modular categories.  As natural generalizations of the Ising anyons \cite{hnw2} they are important examples in the study of topological phases of matter and their applications \cite{hnw} to quantum computation.  They are defined as unitary modular categories with the same fusion rules as those obtained from the semisimple quotients $SO(N)_2$\footnote{This notation is borrowed from conformal field theory.  A more suitable notation might be $Spin(N)_2$ since the objects analogous to the spinor representations are included.} of $\Rep(U_q \mathfrak{so}_N)$ where $q=e^{\pi i/N}$ for $N$ even and $q=e^{\pi i/(2N)}$ for $N$ odd (see \cite{NR} for details of that construction). In general an $N$-metaplectic category has dimension $4N$ and has simple objects of dimension $1,2$ and $\sqrt{N}$ ($N$ odd) or $\sqrt{\frac{N}{2}}$ ($N$ even).  In the case $N$ is odd $N$-metaplectic categories are relative centers of Tambara-Yamagami categories \cite{GNN}. Recently, a complete classification and enumeration of $N$-metaplectic categories has been completed \cite{acrw,BPR,BGPR}.  In addition, the  $N$-metaplectic modular categories coming from quantum groups, i.e. $SO(N)_2$, have been shown to have finite braid group image \cite{RW} verifying the property $F$ conjecture for this subset of metaplectic categories (see \cite{NR}).

In this article we advance our understanding of $N$-metaplectic modular categories in two ways. First we extend the proof of property $F$ from $SO(N)_2$ with $N$ odd to all odd $N$-metaplectic categories. This is achieved as follows. In \cite{acrw} it is shown that for $N$ odd there are precisely $2^{s+1}$ inequivalent $N$-metaplectic categories where $s$ is the number of prime factors of $N$.  We show that each of these may be obtained from $SO(N)_2$ by Galois conjugation and twisting, which then allows us to describe the images of all $N$-metaplectic $\mcB_n$-representations in terms of those obtained from $SO(N)_2$. Although we believe this technique should apply to the even $N$ cases as well, there are some further technicalities that have not been worked out yet. On the other hand our second result shows that even $N$-metaplectic categories appear in towers of gaugings.  More precisely we show that if $8 \mid N$ then any $N$-metaplectic modular category is a $\Z_2$-gauging of an $\frac{N}{4}$-metaplectic modular category.  Thus for each odd $k$ there are towers of even $N$-metaplectic categories starting with the $2k$- and $4k$-metaplectic categories.

\section{Preliminaries}\label{prelim}
We assume the reader is familiar with the basic notions in the theory of fusion categories such as spherical and braiding structures and their properties.  Good references for these details are: \cite{EGNObook,eno,BaKi,TuraevBook}.

\subsection{Galois conjugation and twisting}
It is well known that a fusion (or modular or ribbon) category $\mC$ can be defined over a number field $\FF=\QQ(\alpha)$.  That is, the data needed to construct $\mC$ ($6j$-symbols, braiding isomorphisms, twists, mapping class group representations) all lie in a finite Galois extension of $\QQ$. Moreover, if $\sigma$ is a Galois automorphism of $\FF$ then twisting all data by $\sigma$ produces another category $\mC^\sigma$.  Now if $\mC$ is a unitary category, or (possibly more generally) has dimension function taking values in $\RR^+$ then $\mC^\sigma$ may not have this property.  Indeed, a Galois conjugate of a pseudo-unitary category is not generally pseudo-unitary.

On the other hand, any Galois conjugate of a \emph{weakly integral} fusion category is pseudo-unitary \cite[Proposition 8.24]{eno}.
Thus, by
\cite[Propositions 8.23]{eno} any \emph{weakly integral} fusion category admits a unique spherical structure $j_+$ with respect to which each object has positive dimension.  Moreover, if $\mB$ is the braided fusion category underlying a weakly integral modular category $\mC$ (i.e. forgetting the spherical structure) then $\mB$ equipped with any other choice of spherical structure is again modular (see \cite[Lemma 2.4]{BNRW1}).  In particular, with respect to the unique spherical structure $j_+$ giving $\mB^\sigma$ positive dimensions, $\mB^\sigma_+=(\mB^\sigma,j_+)$ is modular.  Note that while $\mC^\sigma$ and $\mB^\sigma_+$ have the same underlying braided fusion category $\mB^\sigma$, their spherical structures (and therefore $S$ and $T$-matrices) may differ.

These arguments prove the following useful:
\begin{prop} Let $\mC$ be any weakly integral modular category, $\mB$ its underlying braided fusion category, and $\sigma\in Gal(\overline{\QQ}/\QQ)$ a Galois automorphism.  Then there is a unique choice of a spherical structure $j_+$ with respect to which $\mB^\sigma_+=(\mB^\sigma,j_+)$ is a modular category with positive dimensions.
\end{prop}

It is worth pointing out that distinct spherical structures on the braided fusion category $\mB$ underlying any modular category $\mC$ are in 1-1 correspondence with invertible self-dual objects of $\mC$ (see e.g. \cite{EGNObook}).

A motivation for this paper is the following:
\begin{conj} The braid group representations associated with any object in a weakly integral braided fusion category has finite image.
\end{conj}
An object $X\in\mB$ so that the corresponding braid group representations all have finite image is called a \textbf{property F} object, and $\mB$ has property F if all objects are property F objects. It is conjectured (see \cite{NR}) that $\dim(X)^2\in \ZZ$ if and only if $X$ has property F, so that $\mB$ has property F if and only if $\mB$ is weakly integral.

Suppose that every object in a modular $\mC$ has property F.  Then the same is true of $\mC^\sigma$, since the relations defining a finite group are polynomials.  Moreover, the braid group image only depends on the underlying braided fusion category $\mB$, i.e. is independent of the spherical structure.  Thus if a weakly integral modular category $\mC$ has property F then for any Galois conjugation $\sigma$ the underlying braided fusion category, $\mB^\sigma$ equipped with the positive spherical structure $\mB^\sigma_+$ also has property F.

Recently it was shown \cite{RW} that the integral modular categories $SO(N)_2$ obtained from quantum groups $U_q so_N$ at $q=e^{\pi i/N}$ ($N$ even) and $q=e^{\pi i/(2N)}$ ($N$ odd) have property F.  The proof involves a detailed analysis of representations of these quantum groups, rather than categorical-level arguments.  In particular the proof does not immediately imply that unitary modular categories with the same fusion rules as $SO(N)_2$ (i.e. \emph{metaplectic modular categories} also have property F).  On the other hand, metaplectic modular categories have now been classified and enumerated.  This suggests that we can infer property F for those metaplectic modular categories with underlying braided fusion categories Galois conjugate to $SO(N)_2$.

\subsection{Boson Condensation and Gauging}
Two processes that we employ in our analysis are gauging and de-gauging (sometimes called anyon condensation), which may be interpreted physically as phase transitions for anyon systems \cite{Burnell}.  First let us introduce the basic construction we call de-gauging (which was first described in \cite{Pareigis95} and subsequently rediscovered and developed in \cite{MuegerJAlg03,Brugieres,DMNO1} under various conditions and under different names).  Let $\mcC$ be modular and $\Rep(G)\cong \mcD\subset\mcC$ a Tannakian subcategory (here a Tannakian category is a symmetric braided fusion category equivalent to $\Rep(G)$ for some finite group $G$).  The $G$-de-equivariantization $\mcC_G$ of $\mcC$ is a faithfully $G$-graded category (in fact, a braided $G$-crossed category) with modular trivial component $[\mcC_G]_e$ of dimension $\dim(\mcC)/|G|^2$  and $[\mcC_G]_e$ is the \textbf{$G$-de-gauging} of $\mcC$ \cite{DMNO1}.  One does not need to understand the full $G$-de-equivariantization of $\mC$ to obtain $[\mC_G]_e$: in fact $[\mC_G]_e=(\mcD^\prime)_G$, where 
$$\mcD^\prime=\{Y\in\mC: c_{X,Y}c_{Y,X}=\id_{Y\otimes X}\;\text{for all}\; X\in\mcD\}$$ is the M\"uger centralizer of $\mcD\subset\mC$ \cite{DMNO1}.

The simplest case of de-gauging is \textbf{boson condensation}.  Whenever a modular category $\mathcal{C}$ contains a \textbf{boson} $b$,  i.e. a self-dual invertible object with twist $\theta_b=1$, then the fusion subcategory $\langle b \rangle$ is equivalent to  $\Rep(\ZZ_2)$.   In this case, the de-equivariantization functor $F: \mcC \to \mcC_G$ is easier to understand.   In particular, if $X \in \mcC$ is a simple object and $b \otimes X \not\cong X$, then $F(X) \cong X^{(1)} \oplus X^{(2)}$ for simple objects $X^{(1)}, X^{(2)}$.  On the other hand, if $b \otimes X \cong X$, then $F(X)$ is a simple object.  There is a trichotomy among self-dual invertible objects in a ribbon category: they are either bosons as above, \emph{semions} $s$ with $\theta_s=\pm i$ in which case the subcategory $\langle s\rangle$ is modular or \emph{fermions} $f$ with $\theta_f=-1$ and $\langle f\rangle \cong \sVec$.

The reverse process, $G$-gauging, is more complicated \cite{BBCW,CGPW}.  Here one starts with a modular category $\mcB$ and an action of a finite group $G$ by braided tensor autoequivalences: $\rho:G\rightarrow \Aut_{\ot}^{br}(\mcB)$.  A  \textbf{$G$-gauging} of $\mcB$, when it exists, is a new modular category obtained by first constructing a $G$-graded fusion category $\mcD$ with trivial component $\mcD_e=\mcB$ and then equivariantizing to obtain a new modular category $\mcD^G$.  There are obstructions to the existence of a gauging, and when the obstructions vanish there can be many $G$-gaugings (see \cite{CGPW}).  A recent result of Natale \cite{Natale} implies that any weakly \textit{group-theoretical} modular category is a $G$-gauging of either a pointed modular category or a Deligne product of a pointed modular category and an Ising category.  In \cite[Question 2]{ENO2} they ask if every weakly integral modular category is weakly group-theoretical (the converse is known to be true).  If the answer is ``yes'' (as many suspect) then to prove one direction of the property $F$ conjecture it would be enough to prove that $G$-gauging preserves property $F$.

\section{Metaplectic Categories}\label{Metaplectic}
We begin with the following definition:

\begin{definition}
A metaplectic modular category is a unitary modular category with the same fusion rules as $SO(N)_2$ for some $N>1$.
\end{definition}
The structure and properties of $SO(N)_2$ were studied in some detail in \cite{NR}, from which much of the results we outline are taken. 
The fusion rules for $SO(N)_2$ (and hence $N$-metaplectic modular categories) naturally split into three cases, depending on the value of $N$ mod $4$.  

\subsection{Fusion rules for odd $N$.}

The $N$-metaplectic modular categories for odd $N>1$ have $2$ simple objects $X_1, X_2$ of dimension $\sqrt{N}$, two simple objects $\one, Z$ of dimension $1$, and $\frac{N-1}{2}$ objects $Y_i$, $1\leq i\leq\frac{N-1}{2}$ of dimension $2$.  The fusion rules are \cite{acrw}:
\begin{enumerate}
 \item $Z\otimes Y_i\cong Y_i$, $Z\otimes X_i\cong X_{i+1}$ (modulo $2$), $Z^{\otimes 2}\cong\one$,
 \item $X_i^{\otimes 2}\cong \one\oplus \bigoplus_{i} Y_i$,
 \item $X_1\otimes X_2\cong Z\oplus\bigoplus_{i} Y_i$,
 \item $Y_i\otimes Y_j\cong Y_{\min\{i+j,N-i-j\}}\oplus Y_{|i-j|}$, for $i\neq j$ and $Y_i^{\otimes 2}=\one\oplus Z\oplus Y_{\min\{2i,N-2i\}}$.
\end{enumerate}
It is shown in \cite{acrw} that $Z$ is always a boson, and 
 $N$-metaplectic modular categories with $N$ odd were classified and enumerated by condensing $Z$: there are precisely $2^{s+1}$ inequivalent such categories, where $s$ is the number of distinct primes dividing $N$. The fusion rules for the (adjoint) subcategory generated by $Y_1$ with simple objects $\one, Z$ and all $Y_i$ are precisely those of the dihedral group $D_N$ of order $2N$, and, moreover this subcategory coincides the centeralizer of the Tannakian $\langle Z\rangle\cong\Rep(\Z_2)$.

\subsection{Fusion rules for $N \equiv 2 \pmod{4}$.}\label{2mod4}
  The $N$ metaplectic modular categories for $N\equiv 2\pmod{4}$ have
  rank $k+7$, where $k = N/2$ (an odd number). We will denote by $SO(2)_2$ the pointed modular category $\mcC(\Z_8,Q)$ with twists $e^{j^2\pi i/16}$ for uniformity of notation so that there are $4$ inequivalent $2$-metaplectic modular categories (since there are $4$ inequivalent non-degenerate symmetric quadratic forms on $\Z_8$ see \cite{Wall}).  Generally, \cite{BPR} there are exactly $2^{s+1}$ inequivalent $N$-metaplectic modular categories in this case, where $s$ is the number of prime divisors of $N$.  The group of isomorphism classes of invertible objects for $N\geq 6$
  is isomorphic to $\ZZ_4$.     Let $g$ be a generator of this
  group, so the (isomorphism classes of) invertible objects are $g^{j}$ for $0\leq j\leq 3$.  There
  are $k-1$ self-dual simple objects, $X_i$ and $Y_{i}$ for $1\leq i\leq \frac{k-1}{2}$, of dimension $2$.  The remaining four
  simples objects, $V_{i}$ for $1\leq i\leq 4$, have dimension $\sqrt{k}$.  The following fusion rules hold \cite{BPR}:
  \begin{itemize}
    \item $g\otimes X_a\cong Y_{\frac{k+1}{2}-a}$, and $g^2\otimes X_a\cong
    X_a,$ and $g^{2}\otimes Y_{a}\cong Y_{a}$ for  $1\leq a\leq
    (k-1)/2$.
    \item $X_a\otimes X_a \cong \one \oplus g^2\oplus X_{\mathrm{min}\{2a, k-2a\}}$; $X_{a}\otimes X_{b} \cong X_{\mathrm{min}\{a+b,k-a-b\}}\oplus X_{|a-b|}$  ($a\neq b$)
    
    \item $V_1\otimes V_1 \cong g\oplus
    \bigoplus\limits_{a= 1}^{\frac{k-1}{2}}Y_{a}$.
    \item $gV_{1}=V_{3}$, $gV_{3} \cong V_{4}$, $gV_{2}\cong V_{1}$, $gV_{4} \cong V_{2}$ and $g^{3}V_{a} \cong V_{a}^{*}$, $V_{2} \cong V_{1}^{*}$, $V_{4} \cong V_{3}^{*}$
  \end{itemize} Again adopting the same notion for simple objects in a general $N$-metaplectic category $\mC$ with $N\equiv 2\pmod{4}$ one finds that $g^2$ is always a boson and the classification of $N$-metaplectic modular categories with $N\equiv 2\pmod{4}$ was obtained in \cite{BPR} by condensing $\langle g^2\rangle$, to obtain a pointed cyclic modular category.  Indeed, the centralizer of $\langle g^2\rangle\cong\Rep(\Z_2)$ has simple objects $X_i,Y_i$ and the $g^j$ i.e all simple objects of dimension $1$ or $2$.  The simple object $Y_1$ generates this subcategory, which has the same fusion rules as $\Rep(\Z_4\ltimes \Z_{k})$ (with the generator of $\Z_4$ acting by inversion on $\Z_k$) see \cite[Remark 4.4 and Theorem 4.8]{NR}.  In this notation the $\Z_4$-grading on $\mC$ has trivial component $\mC_0$ with simple objects $\one,g^2,X_1,\ldots,X_{\frac{k-1}{2}}$, component $\mC_2$ with simple objects $g,g^3,Y_1,\ldots,Y_{\frac{k-1}{2}}$ and the other two components with simple objects $\{V_1,V_3\}$ and $\{V_2,V_4\}$ respectively.  Obviously there are labeling ambiguities associated with $g\leftrightarrow g^3$ and $\{V_1,V_3\}\leftrightarrow\{V_2,V_4\}$.

\subsection{Fusion rules for $N \equiv 0 \pmod{4}$.}\label{4 divides N}
The $N$-metaplectic modular categories with $N\equiv 0\pmod{4}$  with $2k=N$ have
  rank $k+7$ and dimension $4N$ \cite{NR}.  The simple objects have dimension $1,2$ and $\sqrt{k}$ and are all self-dual.
Setting  $r = \frac{k}{2} - 1$, the ($2r+1=k-1$) simple objects $X_i$ for $0\leq i\leq r-1$ and $Y_j$ for $0\leq j\leq r$ have dimension $2$ and the simple objects $V_i,W_i$ have dimension $\sqrt{k}$.  
 For $k>2$ the key fusion rules are as follows \cite{BGPR}:
   
  \begin{itemize}
  \item $h^{\ot 2} \cong g^{\ot 2} \cong \one$, $h\ot X_i \cong g\ot X_i \cong X_{r-i-1}$ and $h\ot Y_i \cong g\ot Y_i \cong Y_{r-i}$
  \item $g\ot V_1 \cong V_2, h\ot V_1 \cong V_1$ and $h\ot W_1 \cong W_2, g\ot W_1 \cong W_1$
  \item $V_1^{\ot 2} \cong \one\oplus h\oplus\bigoplus_{i=0}^{r-1} X_i$
  \item $W_1^{\ot 2} \cong \one\oplus g\oplus\bigoplus_{i=0}^{r-1} X_i$
  \item $W_1\ot V_1 \cong \bigoplus_{i=0}^r Y_i$
  \item $X_i\ot X_j \cong \begin{cases} X_{i+j+1}\oplus X_{j-i-1} & i<j\leq\frac{r-1}{2}\\ \one\oplus hg\oplus X_{2i+1} & i=j<\frac{r-1}{2}\\ \one \oplus h\oplus g\oplus hg & i=j=\frac{r-1}{2}<r-1\end{cases}$
  \item $Y_i\ot Y_j \cong \begin{cases} X_{i+j}\oplus X_{j-i-1} & i<j\leq\frac{r}{2}\\ \one\oplus hg\oplus X_{2i} & i=j\leq\frac{r-1}{2}\\ \one\oplus h\oplus g\oplus hg &i=j=\frac{r}{2}.\end{cases}$
 
  \end{itemize}

Notice that all other fusion rules may be derived from the above by tensoring with $h$ or $g$ as needed.  For example $V_1\ot V_2 \cong g\ot V_1^{\ot 2} \cong h\oplus hg\oplus\bigoplus_{i=0}^{r-1}X_i$.  The $\Z_2\times\Z_2$ grading is clear from these rules, we denote the trivial component by $\mC_{(0,0)}$ and the component with simple objects $Y_j$ by $\mC_{(1,1)}$.  The classification of $N$-metaplectic modular categories with $4\mid N$ was obtained in \cite{BGPR} by condensing $hg$, which is always a boson.   It is shown in \cite{BGPR} that, for $N\geq 8$ there are $3\cdot 2^{s+1}$ inequivalent $N$-metaplectic modular categories where $s$ is the number of distinct primes dividing $N$.  The degenerate case $N=4$ is special: it has fusion rules like $\Ising^\nu\boxtimes\Ising^\mu$ for which there are $20$ inequivalent metaplectic modular categories, rather than $12$.

The centralizer of the pointed subcategory $\langle h,g\rangle$ is always the trivial component $\mC_{(0,0)}$ with simple objects $\one,h,g,hg$, and all $X_i$, whereas $\langle hg\rangle^\prime$ also includes the component $\mC_{(1,1)}$ with simple objects $Y_j$ and the component with simple objects $V_i$ by $\mC_{(1,0)}$ for concreteness. There is a slight further subtlety related to the value of $N\equiv 0,4\pmod{8}$. The objects $h,g$ are bosons precisely when $8\mid N$, and are fermions otherwise.  Moreover, when $8\mid N$ one sees that $h$ centralizes the trivial component as well as the component $\mC_{(1,0)}$ containing $V_1$ and $V_2$, while $g$ centralizes the $W_i$.  When $8\nmid N$ the opposite is true: $g$ centralizes the $V_i$ and $h$ centralizes the $W_i$ \cite{BGPR}. In \cite{NR} it is shown that the fusion subcategory generated by $Y_0$ (i.e. $\langle hg\rangle^\prime=\mC_{(0,0)}\oplus\mC_{(1,1)}$) has the same fusion rules as the representation category $\Rep(D_{N})$ of the dihedral group of order $N$.

\section{Property F for $N$-Metaplectic Categories with $N$ odd}

% In this section we verify property F for two extreme cases of $N$-metaplectic modular categories: 1) for $N$ an odd square or twice a square and 2) for $N$ odd.

% Recall from \cite{ERW} that group-theoretical braided fusion categories enjoy property F.
% \begin{theorem} Suppose that $\mC$ is a $N$-metaplectic modular category with $N=k^2$ for some odd $k$ or $N=2\ell^2$.  Then $\mC$ is group-theoretical, and hence has property F.
% \end{theorem}

% \begin{proof} By \cite[Corollary 4.14]{0704.0195} it is enough to show that $\mC$ has a symmetric subcategory $\mcL$ such that $(\mcL^\prime)_{ad}\subset \mL$.  For $N=k^2$ odd, we take $\mcL$ to be generated by the simple objects $\one,b,X_k,\cdots X_{tk}$ 

% For $N$ even FILL IN DETAILS FOLLOWING Naidu-Rowell.  
% \end{proof}

\begin{theorem} If $\mC$ is an  $N$-metaplectic modular category with $N := 2r+1$ odd, then $\mC$ has property F.
\end{theorem}
\begin{proof}

Let $N = p_1^{a_1} \cdots p_s^{a_s}$ be the prime factorization of $N$. From \cite{acrw}, we know that there are precisely $2^{s+1}$ $N$-metaplectic modular categories.  We will show that Galois conjugation and twisting \cite{fermion16} produce all of these categories.

A Galois conjugate of the quantum group category $SO(N)_2$ is not necessarily unitary.  However, it is pseudounitary, so there exists a choice of spherical structure on its underlying braided fusion category to make it unitary.  This choice does not affect the braiding eigenvalues of the category.  

Let $\zeta = e^{\frac{2 \pi i}{16r + 8}}$.  There exists a simple object $W \in SO(N)_2$ of dimension $\sqrt{N}$ such that the eigenvalues of the braiding $R_{W,W}$ are $\zeta^{n_j}$ for $n_j =  (4r + 2)((r - j)(r - j + 1) - j) + (2r + 1)r + 2j^2, $ and $0 \le j \le r$ \cite{hnw}.  The non-isomorphic simple object $W'$ of dimesion $\sqrt{N}$ has braiding eigenvalues $-\zeta^{n_j}$ for $0 \le j \le r$.

The Galois group $\Gal(\QQ(\zeta)/\QQ) \cong (\ZZ/8N\ZZ)^\times$ acts on the set of eigenvalue exponents $\{n_j : 0 \le j \le r\} \subset \ZZ/8N\ZZ$ by left translation.  By the Chinese Remainder theorem, the Galois group acts on each factor of $\ZZ/8N\ZZ \cong \ZZ/p_1^{a_1}\ZZ \times \cdots \times \ZZ/p_s^{a_s}\ZZ \times \ZZ/8\ZZ$ independently.  

We first observe that $n_j = 2j^2 \pmod{N}$.  Since $2(-j)^2 = 2j^2$, we have $\{n_j \pmod{N} : 0 \le j \le r\} = \{2j^2 : j \in \ZZ/N\ZZ \}$  as sets.  Hence, for any $i$, we have $X := \{n_j \pmod{p_i^{a_i}} : 0 \le j \le r\} = \{2j^2 : j \in \ZZ/p_i^{a_i}\ZZ \}$.  The factor of the Galois group acting on $\ZZ/p_i^{a_i}\ZZ$ is $(\ZZ/p_i^{a_i})^\times$.  Since $(\ZZ/p_i^{a_i}\ZZ)^\times$ is cyclic, the stabilizer subgroup $\Stab_{(\ZZ/p_i^{a_i}\ZZ)^\times}(X) = \{x^2: x \in (\ZZ/p_i^{a_i}\ZZ)^\times\}$ has index $2$.    Thus, we get two distinct sets of eigenvalues mod $p_i^{a_i}$ for each $i$.  

%We claim that if $ k \neq l \pmod{N}$, then $\zeta^{n_k} \neq \zeta^{n_l}$.  We have $n_j = 2 j^2 \pmod{N}$ for $0 \le j \le r$.  Hence, $n_k - n_l = 2(k^2 - l^2) = 2(k-l)(k+l) \pmod{N}$. Since $k,l \le r < N/2$, this difference is nonzero.   Hence, $\zeta^{n_k} \neq \zeta^{n_l}$.  

%The Galois group $\Gal(\QQ(\zeta)/\QQ) \cong (\ZZ/N\ZZ)^\times$ is a product of cyclic groups of even order.

%Since $N$ is an odd prime, the group $(\ZZ/N\ZZ)^\times$ is cyclic of order $2r$.  Let $\phi: (\ZZ/N\ZZ)^\times \to (\ZZ/N\ZZ)^\times$ be defined by $\phi(g) = g^2$, and let $H = \Im(\phi)$ be the  subgroup of even powers. For $1 \le j \le r$, we have $[n_j]_N \in 2H$.  The order of the coset $2H$ is precisely $r$, so $2H = \{[n_j]_N : 1 \le j \le r \}$.  

% The Galois group $\Gal(\QQ(\zeta)/\QQ)$ acts on $(\ZZ/N\ZZ)^\times$ by translation and fixes $\zeta^{n_0} = 1$. This contributes a factor of 2 for number of possible eigenvalue sets, corresponding to the two cosets of $H$.

Moreover, we have $n_j = r \pmod{8}$ for all $j$.  If $r$ is relatively prime to $8$, this gives 4 choices of Galois conjugates for $[n_j]_8$.  If $r = 2$ or $r = 6$ mod 8, we have 2 choices.  If $r = 0$ or $r = 4$ mod 8, there is only one choice.     In all but the last ($r = 0$ or $r = 4$) case, we must divide by $2$ to account for labelling ambiguity on the nonintegral objects.  Thus, when $r$ is relatively prime to $8$, we get $(2^s)(4)/2 = 2^{s+1}$ distinct categories from Galois conjugation.  When $r = 2$ or $r = 6$ mod 8, we get $(2)(2)/2 = 2^s$ distinct categories.  When $r = 0$ or $r = 4$ mod 8, we get $(2)(1) = 2^s$ modular categories.

\newcommand{\wh}{\widehat}

To construct the remaining metaplectic modular categories, we will use twisting in the sense of Bruillard et al. \cite{fermion16}. Let $\mD$ be a modular category.  Let $B \subset G(\mD)$ be a subgroup of the group of the invertibles of $\mD$, and let $w \in Z^3(\wh{B}, U(1))$ be a 3-cocycle.  The twisted category $\mD_{(1,w)}$ is a $\wh{B}$-graded category with the same objects and tensor product as $\mD$, but with an associator twisted by $w$.  More explicitly, if $\sigma, \tau, \rho \in \widehat{B}$, then we have
$$\wh \alpha_{X_\sigma, X_\tau, X_\rho} = w_{\sigma, \tau, \rho} \alpha_{X_\sigma, X_\tau, X_\rho},$$
where $\wh \alpha $ and $\alpha$ are the associators of $\mD_{(1,w)}$ and $\mD$, respectively.

Let $B \subset G(B)$ be a subgroup such that the induced map $U(G) \to \wh{G(B)} \to \wh B \cong \ZZ_2$ corresponds to the GN-grading. Let $w \in Z^3(\ZZ_2, U(1))$ be the normalized 3-cocycle  given by $w(1,1,1) = -1$.  Let $\alpha$ and $c$ denote the associator and braiding for some metaplectic modular category $\mD$, and let $\widehat\alpha$ and $\widehat c$ denote the associator and braiding of the twisted category  $\mD_{(1,w)}$, respectively. 

We claim that a solution to the hexagon equations is given by 
$$\widehat{c}_{X_\sigma, X_\tau} = {\epsilon_{\sigma, \tau}} c_{X_\sigma, X_\tau},$$
where $\epsilon_{\sigma,\tau} = i$ if $\sigma = \tau = 1$, and $\epsilon_{\sigma,\tau} = 1$ otherwise.  Indeed,  in diagrammatic composition order, we have
\begin{align*}
& \widehat\alpha_{X_\sigma, X_\tau, X_\rho} \circ \widehat c_{X_\sigma,X_\tau \otimes X_\rho} \circ \widehat \alpha_{X_\tau,X_\rho,X_\sigma} \\
& \qquad = (w_{\sigma,\tau,\rho}\alpha_{X_\sigma,X_\tau,X_\rho}) \circ  \epsilon_{\sigma,\tau\rho} c_{X_\sigma,X_\tau \otimes X_\rho} \circ (w_{\tau,\rho,\sigma} \alpha_{X_\tau,X_\rho,X_\sigma}) \\
& \qquad = \epsilon_{\sigma,\tau\rho} \alpha_{X_\sigma,X_\tau,X_\rho} \circ  c_{X_\sigma,X_\tau \otimes X_\rho} \circ \alpha_{X_\tau,X_\rho,X_\sigma} \\
& \qquad = \epsilon_{\sigma,\tau\rho} \cdot (c_{X_\sigma,X_\tau} \otimes \id_{X_\rho}) \circ \alpha_{X_\tau,X_\sigma,X_\rho} \circ (\id_{X_\tau} \otimes c_{X_\sigma,X_\rho}) \\
& \qquad = \epsilon_{\sigma,\tau\rho} \epsilon_{\sigma,\tau}^{-1} \epsilon_{\sigma,\rho}^{-1} w_{\tau,\sigma,\rho} \cdot (\widehat c_{X_\sigma,X_\tau} \otimes \id_{X_\rho}) \circ (\widehat \alpha_{X_\tau,X_\sigma,X_\rho}) \circ (\id_{X_\tau} \otimes \widehat c_{X_\sigma,X_\rho}) \\
& \qquad =  (\widehat c_{X_\sigma,X_\tau} \otimes \id_{X_\rho}) \circ (\widehat \alpha_{X_\tau,X_\sigma,X_\rho}) \circ (\id_{X_\tau} \otimes \widehat c_{X_\sigma,X_\rho}), 
\end{align*}
where the last equality follows from case analysis. The verification for the other hexagon equation is analogous.

The spherical structure on the twisted category $\mD_{(1,w)}$ is the same as the spherical structure on $\mD$.  Since $\epsilon$ and $w$ are $U(1)$-valued, the modular category $\mD_{(1,w)}$ is also unitary.

Since any matrix in the twisted braid group representation differs from a matrix in the untwisted representation by a factor of the form $i^n$, this twisting preserves Property F.  By examining the exponents of the braiding eigenvalues mod $8$, we find that twisting accounts for another factor of $2$ in our count when $r$ is even, covering the remaining modular categories.  
\end{proof}

We illustrate the proof of the theorem with the following tables of braiding eigenvalues for $3$- and $5$-metaplectic categories.

\textbf{ $3$-metaplectic categories.}
The following table gives the exponents of the relevant braiding eigenvalues of the Galois conjugates of $SO(3)_2$.  More explicitly, given $\sigma \in  (\ZZ/24\ZZ)^\times \cong \Gal(\QQ(\zeta)/\QQ)$ and $n \in \ZZ/24\ZZ$, we have the group action $\sigma(n) = \sigma \cdot n$.  Letting $\zeta = e^{\frac{\pi i}{12}}$,  the braiding eigenvalues of the $\sigma$-Galois conjugate of the first nonintegral object are $\sigma(R_{V_1, V_1}^i) = \zeta^{\sigma(n_i)}$.  The braiding values of the other nonintegral object are given by $\sigma(R_{V_2, V_2}^i) = -\sigma(R_{V_1, V_1}^i) = \zeta^{\sigma(12 + n_i)}$.
Since $n_0 = 9$ and $n_1 = 1$, we have the following table of exponents of braiding eigenvalues of Galois conjugates.
$$
    \begin{tabular}{L|LLLL}
    \sigma  & \sigma(n_0)  & \sigma(n_1)  & \sigma(12 + n_0) &  \sigma(12 + n_1) \\
    \hline\vrule height 12pt width 0pt
    1       & 9          &  1      &  21   &  13  \\
    5       & {21}       &  5     &   9    &  17 \\
    7       & {15}       &  7     &   3   & 19  \\
    11       & {3}       &  11    &   15   &  23   \\
    \end{tabular} 
$$

Since we know there are precisely four $3$-metaplectic categories, this table illustrates the fact that all four $3$-metaplectic categories lie in the same orbit under the Galois conjugation action, since they are distinguished by these eigenvalues.

\textbf{$5$-Metaplectic categories.}
Here $\zeta = e^{\frac{\pi i}{20}}$.  Similarly, we have the following table of exponents of braiding eigenvalues.

$$
    \begin{tabular}{L|LLLLLL}
    \sigma & \sigma(n_0)  & \sigma(n_1)  & \sigma(n_2) & \sigma(20 + n_0) &  \sigma(20 + n_1)  & \sigma(20 + n_2)\\
    \hline\vrule height 12pt width 0pt
    1       & {10} & {18} &  2 & 30 & 38 & 22 \\
    3       & {30} & {14} &  6 & 10 & 34 & 26 \\
    7       & {30} & {6} &  {14} & 10 & 26 & 34 \\
    9       & {10} & {2} &  {18} & 30 & 22 & 38 \\
    11       & {30} & {38} &  {22} & 10 & 18 & 2\\
    13       & {10} & {34} &  {26} & 30 & 14 & 6 \\
    17       & {10} & {26} &  {34} & 30 & 6 & 14\\
    19       & {30} & {22} &  {38} & 10 & 2 & 18\\
    \end{tabular} 
$$

Since $r = 2$, we only have two distinct sets of braiding eigenvalues in the table, so that Galois conjugation only provides two of the four $5$-metaplectic categories.  The other two categories are obtained by twisting: at the level of eigenvalues this is manifested by twisting by $i$, i.e. adding $10$ to each exponent in a row of the table.

\section{A Sequence of Gaugings}
$N$-metaplectic modular categories with $4\mid N$ have $4$ self-dual invertible objects, are are therefore $\Z_2\times\Z_2$-graded.  The $(0,0)$-graded component is the adjoint subcategory, and without loss of generality we may assume that the $(1,1)$-graded component contains all of the remaining $2$-dimensional simple objects.  The $(1,0)$- and $(0,1)$-graded  components each contain two isomorphism classes of $\sqrt{N/2}$-dimensional simple objects.

When $8\mid N$ the $N$-metaplectic modular categories of  have 3 bosons $hg,h,g$, i.e. invertible, self-centralizing objects with trivial twists.  The centralizer of each of these bosons consists of the $(0,0)$-graded (adjoint) component and one of the three other components.  It was shown in \cite{BGPR} that condensing the boson  $bhg$ that centralizes the (integral) $(1,1)$-graded component yields a cyclic modular category of the form $\mC(\Z_N,q)$, for some non-degenerate symmetric quadratic form $q$ on $\Z_N$.  
Except for the degenerate case $N=8$, the bosons $h,g$ are uniquely determined (up to the labeling ambiguity $h\leftrightarrow g$) by the condition that they centralize a simple object of dimension $\sqrt{\frac{N}{2}}$.  For $N=8$ all non-invertible simple objects have dimension $2$, and the labels of all 3 bosons are ambiguous, i.e. one cannot distinguish them by any of their properties.  The follow shows that condensing either of the two bosons $h,g$ yields another metaplectic modular category.
\begin{theorem}
Let $\mC$ be an $N$-metaplectic modular category with $8 \mid N$, and let $\mD$ be the unitary modular category given by condensing the boson $h \in \mC$ (or $g$) in the notation of subsection \ref{4 divides N}.  Then $\mD$ is an $\frac{N}{4}$-metaplectic modular category.
\end{theorem}
\begin{proof} For the moment, assume that $N\geq 16$. It is relatively straightforward to verify that $\mD$ has the right rank and dimensions of simple objects.  Set $N=2k$ and $r=\frac{k}{2}-1$. $\mC$ has rank $k+7$, with $k-1=2r+1$ objects of dimension $2$: $X_0,\ldots, X_{r-1}$ and $Y_0,\ldots,Y_{r}$.  
By definition $\mD = (\langle h\rangle')_{\Z_2}$ where $\Rep(\Z_2)=\langle h\rangle$.  From the discussion in Section \ref{prelim} we see that $\langle h\rangle^\prime= \mC_{(0,0)}\oplus\mC_{(1,0)}$ with simple objects $$\{\one, h, hg,g, X_0, X_1, \ldots, X_{r-1}, V_1, V_2 \}.$$  Let $F:\langle h \rangle^\prime \to (\langle h\rangle')_{\Z_2}$ be the de-equivariantization functor.  As $h\otimes V_i\cong V_i$ and $h\otimes X_i\cong X_{r-i-1}$ we have the following, where we set $t=\frac{r-1}{2}=\frac{N}{8}-1$:
\begin{enumerate}
\item
$F(V_i)=V_i^{\pz}\oplus V_i^{\po}$, where $V_i^{(j)}$ are $\sqrt{\frac{N}{8}}$-dimensional objects.
\item $\tY_i:=F(X_{2i})\cong F(X_{r-2i-1})$ are simple objects of dimension 2, for $0\leq i< t/2$ (provided $N\geq 16$, otherwise there are no $\tY_i$)
\item $\tX_j:=F(X_{2j+1})\cong F(X_{r-2j-2})$ are simple objects of dimension 2, for $0\leq j <t/2$ (provided $N\geq 24$, otherwise there are no $\tX_j$)
\item $F(X_{t})=g_1\oplus g_2$ with $g_1,g_2$ invertible,
\item $F(h)=F(\one)=\one_{\mD}$ 
\item $F(hg)=F(g)=Z$ an invertible object.
\end{enumerate}

In particular, the modular category $\mD=F(\langle h\rangle^\prime)$ has the same dimensions ($1$ with multiplicity $4$, $2$ with multiplicity $t=\frac{N}{8}-1$ and $\sqrt{\frac{N}{8}}$ with multiplicity $4$), global dimension ($N$) and rank ($\frac{N}{8}+7$) as an $\frac{N}{4}$-metaplectic modular category. It is important to point out that when $16\mid N$  we have $\frac{N}{4}\equiv 0\pmod {4}$ so that $t$ is odd, while $\frac{N}{4}\equiv 2\pmod{4}$ so that $t$ is even otherwise, so these cases correspond to either  the self-dual fusion rules of subsection \ref{4 divides N} or the non-self-dual fusion rules of subsection \ref{2mod4}.  Here are a few useful observations that can be deduced from the fusion rules of $\mC$:

\begin{itemize}
\item $\mD$ is graded by a group of order $4$, with each component of dimension $\frac{N}{4}$.  
\item If $16\mid N$ (so that $t$ is odd) the trivial component $\mD_0$ contains all $1$-dimensional simple objects and $\frac{t-1}{2}$ simple objects of dimension $2$, otherwise (i.e. $t$ is even) the trivial component contains $\one_{\mD}$ and $Z$ but not $g_1$ or $g_2$.  
\item The object $\tY_0$ generates the subcategory with simple objects $\one_{\mD},Z,g_1,g_2$ and all $\tX_j,\tY_i$.
\item The objects $Z,\tY_i,\tX_j$ are self-dual.
\item The subcategory generated by $\tX_0$ is the adjoint subcategory $\mD_0$.  In particular no $\tY_i$ lie in the adjoint subcategory.
\item The 4 objects $V_i^{(j)}$ appear in two distinct graded components, in pairs.

\end{itemize}
 One may directly show that $\mD$ has the same fusion rules as $SO(\frac{N}{4})_2$ using standard techniques, however this is a somewhat tedious task. We will instead make use of \cite[Theorem 4.2 and Remark 4.4]{NR} and the descriptions in Section \ref{Metaplectic} to derive the result.  The first step is to verify that the fusion rules for the $\frac{N}{2}$-dimensional subcategory $\langle \tY_0\rangle$ with simple objects of dimensions $1$ and $2$ has the fusion rules of either $\Rep(D_{\frac{N}{4}})$ for $t$ odd or $\Rep(\Z_4\ltimes \Z_{\frac{k}{4}})$ for $t$ even.  Then we must verify the fusion rules involving the $V_i^{(j)}$ are also as expected.  We will do these tasks simultaneously.
  
For $t$ odd, the observations above reduce the verification of the hypotheses of Theorem 4.2 of \cite{NR} to showing that the $g_i$ are self-dual, from which we can conclude that $\langle \tY_0\rangle$ has the same fusion rules as $\Rep(D_{\frac{N}{4}})$.  For $t$ even, we must show that $g_1\cong g_2^*$ to verify the hypotheses of Remark 4.4 of \cite{NR} to conclude that $\langle \tY_0\rangle$ has the same fusion rules as $\Rep(\Z_4\ltimes \Z_{\frac{k}{4}})$. We calculate:
 \begin{equation}\label{eq1}
 F(V_1^{\ot 2})=\left(V_1^{\pz}\oplus V_1^{\po}\right)^{\ot 2}\cong g_1\oplus g_2\oplus2(\one_{\mD}\bigoplus_{i=0}^{\frac{t-1}{2}}  \tY_i\oplus\bigoplus_{j=0}^{\frac{t-3}{2}}\tX_j).
\end{equation}
Since 
$$\left(V_1^{\pz}\oplus V_1^{\po}\right)^{\ot 2}\cong {\left(V_1^{(1)}\right)}^{\otimes 2} \oplus {\left(V_1^{(2)}\right)}^{\otimes 2}\oplus 2(V_1^{(1)}\ot V_1^{(2)})$$ it is clear that the $g_1,g_2$ cannot be subobjects of $(V_1^{(1)}\ot V_1^{(2)})$ and $V_1^{(j)}$ for $j=0,1$ are either self-dual or dual to each other.  As we have a labeling choice we may assume $g_j\subset \left(V_1^{(j)}\right)^{\otimes 2}$ for $j=0,1$. 

Now observe that  ${\left(V_1^{(j)}\right)}^{\otimes 2}$ is odd-dimensional when $t$ is even so that in this case $\one_{\mD}$ is not a subobject of ${\left(V_1^{(j)}\right)}^{\otimes 2}$ and hence the $V_1^{(j)}$ are non-self-dual, i.e. are dual to each other for $j=0,1$.  Moreover, the $g_i$ are not in the trivially graded component for $t$ even so that we can conclude that the grading is by $\Z_4$ in this case, so that the group of (isomorphism classes of) invertible objects is isomorphic to $\Z_4$ and hence $g_1\cong g_2^*$.  Thus we can conclude that the fusion rules are the same as those of $\Rep(\Z_4\ltimes \Z_k)$. Since the adjoint subcategory $\mD_0$ contains only simple objects $\one_{\mD},Z$ and all $\tX_j$, the fusion rules involving $V_1^{(j)}$ (and similarly $V_2^{(j)}$) are completely determined.

When $t$ is odd ${\left(V_1^{(j)}\right)}^{\otimes 2}$ is even-dimensional so we must have both $\one_{\mD}$ and $g_j$ as subojects.  In particular, the grading is by $\Z_2\times \Z_2$ so that the $g_i$ are self-dual.  Now we can conclude that the fusion rules of the subcategory $\langle \tY_0\rangle$ are the same as $\Rep(D_{\frac{N}{2}})$ and the fusion rules involving $V_{1}^{(j)}$ (and similarly $V_2^{(j)}$) are determined from equation (\ref{eq1}).

Condensing the boson $h$ in an $8$-metaplectic modular category produces a pointed category of dimension $8$, with the same fusion rules as $\Z_8$, which we have conveniently identified with a $2$-metaplectic modular category. 
\end{proof}

\section{Conclusions and Speculations}

We have obtained two results on metaplectic modular categories.  For odd $N$, we extend the results of \cite{RW} proving property $F$ for $SO(N)_2$ to all $N$-metaplectic modular categories.  This provides some insight into the relationships among (certain) braided fusion categories with the same fusion rules.  A recent paper of Nikshych \cite{nik18} explores the different braidings that a fixed fusion category may have.  One consequence (see \cite[Remark 4.2]{nik18}) is that if a \textit{modular} category has property $F$ then any braiding on the underlying fusion category has property $F$ as well (whether the braiding in non-degenerate or not).  Of course, a fixed \textit{unitary} fusion category has a unique \emph{unitary} braiding by results of \cite{Galindounitary}, so for metaplectic categories this does not help. On the other hand, it seems to be often the case that all (finitely many, by Ocneanu rigidity \cite{eno}) fusion categories with a fixed set of fusion rules are related to each other by some type of twisting of associativity constraints (see \cite{KazWenzl,TubaWenzlBCD}, for example).  One conceptual step towards proving property $F$ would be to extend the results of \cite{nik18} to prove that braided fusion categories with a fixed set of fusion rules either all have property $F$ or all do not.

In a related direction, we have shown that the $N$-metaplectic modular categories for $8\mid N$ are obtained from $2k$- and $4k$-metaplectic modular categories (with $k\geq 1$ odd) by iteratively gauging by a non-trivial $\Z_2$-action.  Physically, this can be interpreted to mean that the systems modeled by $2^tk$-metaplectic modular categories for all $s\geq 1$ of the same parity are just different phases of the same topological order \cite{Burnell}.  It is interesting to note that the number of $2^tk$-metaplectic modular categories stabilizes for $t\geq 2$, so that the choices in the $\Z_2$-gauing process are eventually unique. Of course it is already known that any $N$-metaplectic modular category is a $\Z_2$-gauging of a pointed category (\cite{acrw,BPR,BGPR}), but this result provides an infinite sequence of categories with non-trivial Picard group (see \cite{ENO3}), i.e. non-trivial braided tensor autoequivalences.  Notice this is in contrast to the odd $N$-metaplectic modular categories: for example $N=3$ we see that $3$-metaplectic modular categories admit no non-trivial braided tensor autoequivalences.  This can be deduced from \cite{E-M}: the Brauer-Picard group of $SO(3)_2=SU(2)_4$ is $\Z_2$, with the non-trivial element corresponding to interchanging the two $\sqrt{3}-$dimensional objects.  Since the twists of these two objects are distinct, this action does not give a braided tensor autoequivalence.  Of course, failing to have a non-trivial Picard group does not preclude a category from having non-trivial (i.e. not a Deligne product) gaugings: the pointed category $\Sem$ has trivial Picard group and yet prime modular categories of the form $\mC(\Z_8,Q)$ can be obtained as $\Z_2$-gaugings of $\Sem$ \cite{BBCW}.

In the special case when $N=2^k$ we encounter degenerate (in the sense of Lie algebras) categories : an $8$-metaplectic modular category with fusion rules like $SO(8)_2$ has 3 non-trivial bosons, but they cannot be distinguished.  Condensing any of them yields a category with fusion rules like $\Z_8$ which is a $2$-metaplectic category.  If we condense the boson in any of the four $\mC(\Z_8,Q)$ theories we obtain either $\Sem$ or $\overline{\Sem}$, which we could call $\frac{1}{2}$-metaplectic.  It is worth pointing out that $SO(8)_2$ should have an $S_3$ action by braided tensor autoequivalences.  

For $N=16$ if we condense either of the two bosons that centralize a simple object of dimension $\sqrt{8}$ we obtain a $4$-metaplectic modular category, of the form $\Ising^\nu\boxtimes\Ising^\mu$ e.g. $SO(4)_2$.  It is known (\cite{BGPR}) that there are 12 inequivalent $16$-metaplectic modular categories, whereas there are $20$ with the same fusion rules as $\Ising\boxtimes\Ising$.  Which of the $20$ can appear in this way?  In this case we find that only the $12$ that are $\Z_2$-gaugings of the $4$ pointed categories $\mC(\Z_4,Q_s)$ have the correct central charge $e^{(2s+1)\pi i/4}$.  We could call these $\mC(\Z_4,Q_s)$ theories $1$-metaplectic categories as they are obtained by condensing a boson in $SO(4)_2$.  
More generally, let $k\geq 3$ be an odd number with precisely $s$ distinct prime factors.  Then there are $2^{s+2}$ inequivalent $2k$-metaplectic categories \cite{BPR} and $3\cdot 2^{s+2}$ inequivalent $2^ak$-metaplectic categories for $a\geq 2$ \cite{BGPR}.  In particular we find that the cohomological choices in the gauging process from a $2^ak$-metaplectic category to a $2^{a+2}k$-metaplectic category does not increase the number of such categories, rather their cardinality stabilizes.

\end{document}